\documentclass[11pt]{amsart}

\usepackage[T1]{fontenc}
\usepackage[utf8]{inputenc}
\usepackage{lmodern}
\usepackage{amsmath,amssymb,amsfonts,amsthm}
\usepackage{aliascnt}
\usepackage{mathtools}
\usepackage{mathrsfs}
\usepackage{microtype}

\usepackage[margin=1in]{geometry}
\usepackage[shortlabels]{enumitem}
\setlist{listparindent=\parindent, parsep=0pt}
\usepackage[dvipsnames]{xcolor}
\usepackage{hyperref}
\hypersetup{linktocpage,colorlinks,linkcolor=blue,citecolor=magenta,urlcolor=blue}
\usepackage[capitalise]{cleveref}
\crefformat{equation}{(#2#1#3)}
\allowdisplaybreaks

\theoremstyle{plain}
\newtheorem{thm}{Theorem}[section]
\newaliascnt{prop}{thm}
\newtheorem{prop}[prop]{Proposition}
\aliascntresetthe{prop}
\newaliascnt{lemma}{thm}
\newtheorem{lemma}[lemma]{Lemma}
\aliascntresetthe{lemma}
\newaliascnt{corollary}{thm}

\aliascntresetthe{corollary}

\theoremstyle{definition}
\newaliascnt{mydef}{thm}
\newtheorem{mydef}[mydef]{Definition}
\aliascntresetthe{mydef}
\newaliascnt{remark}{thm}
\newtheorem{remark}[remark]{Remark}
\aliascntresetthe{remark}

\crefname{thm}{Theorem}{Theorems}
\Crefname{thm}{Theorem}{Theorems}
\crefname{prop}{Proposition}{Propositions}
\Crefname{prop}{Proposition}{Propositions}
\crefname{lemma}{Lemma}{Lemmas}
\Crefname{lemma}{Lemma}{Lemmas}
\crefname{corollary}{Corollary}{Corollaries}
\Crefname{corollary}{Corollary}{Corollaries}
\crefname{mydef}{Definition}{Definitions}
\Crefname{mydef}{Definition}{Definitions}
\crefname{remark}{Remark}{Remarks}
\Crefname{remark}{Remark}{Remarks}

\numberwithin{equation}{section}

\DeclareMathOperator*{\argmin}{argmin}

\newcommand{\Sd}{\mathbb S^{d-1}}
\newcommand{\R}{\mathbb R}

\newcommand{\sg}{\sigma}
\newcommand{\Pc}{\mathcal P}
\newcommand{\Hh}{\mathcal H}
\newcommand{\M}{\mathcal M}
\newcommand{\F}{\mathcal F}
\newcommand{\E}{\mathcal E}
\newcommand{\Ent}{\mathrm H}
\newcommand{\ip}[2]{\left\langle #1,#2\right\rangle}
\newcommand{\norm}[1]{\left\lVert #1\right\rVert}
\newcommand{\ep}{\varepsilon}

\newcommand{\qu}{q_u}

\title[Phase transitions for the noisy transformer model]{Phase transitions for the noisy transformer model in arbitrary dimension}

\author[K. Mun]{Kyunghoo Mun}
\address{Kyunghoo Mun, Department of Mathematical Sciences, Carnegie Mellon University, Pittsburgh, PA 15213, USA}
\email{kmun@andrew.cmu.edu}
\thanks{K.M. was supported by NSF grant DMS-2441170.}

\author[M. Rosenzweig]{Matthew Rosenzweig}
\address{Matthew Rosenzweig, Department of Mathematical Sciences, Carnegie Mellon University, Pittsburgh, PA 15213, USA}
\email{mrosenz2@andrew.cmu.edu}
\thanks{M.R. was supported by NSF grants DMS-2441170, DMS-2345533, DMS-2342349.}

\begin{document}

\begin{abstract}
We study the McKean--Vlasov free energy on the unit sphere associated with the unnormalized self-attention (USA) model for noisy transformer dynamics.  We prove a sharp global-minimizer dichotomy in every dimension $d\ge2$.  There is a unique $\beta_*^{(d)}>0$ such that
\begin{equation*}
        \frac{I_{d/2+1}(\beta_*^{(d)})}{I_{d/2}(\beta_*^{(d)})}=\frac1d,
\end{equation*}
where $I_\nu$ is the modified Bessel function of the first kind.  For $0<\beta\le \beta_*^{(d)}$, the uniform density remains the unique global minimizer up to the linear-stability threshold
\begin{equation*}
        K_\#^{(d)}(\beta)=\frac{\beta^{d/2}}{2^{d/2}\Gamma(d/2)I_{d/2}(\beta)},
\end{equation*}
and the phase transition is continuous.  For $\beta>\beta_*^{(d)}$, the uniform density is not globally minimizing at $K_\#^{(d)}(\beta)$, so the critical coupling satisfies $K_c<K_\#^{(d)}(\beta)$ and the transition is discontinuous.  This result generalizes the authors' recent $d=2$ work \cite{MunRosenzweig2026Circle} to arbitrary dimension.  The proof uses the sharp Beckner--Onofri/logarithmic Hardy-Littlewood-Sobolev (HLS) inequality on the sphere, together with a Funk--Hecke/Bessel coefficient computation and a degree-two quartic obstruction.
\end{abstract}

\maketitle

\section{Introduction}\label{sec:introduction}

Transformers are neural-network architectures in which self-attention updates each token by comparing it with the other tokens through query--key inner products and then aggregating the corresponding value vectors \cite{VaswaniEtAl2017}.  Recent mathematical work isolates this mechanism by considering deep, continuous-time limits of attention layers \cite{GeshkovskiLetrouitPolyanskiyRigollet2025} (see also the works \cite{ChenRubanovaBettencourtDuvenaud2018,E2017Proposal,HaberRuthotto2017,LuLiHeSunDongQinWangLiu2020,DuttaGautamChakrabartiChakraborty2021,SanderAblinBlondelPeyre2022} for a machine learning perspective).  After layer normalization, token embeddings may be viewed as directions on a sphere, and the attention update becomes an interacting particle system on $\Sd$; its large-token, or mean-field, limit is a McKean--Vlasov equation, alternatively a Wasserstein gradient flow of a certain free energy \cite{GeshkovskiLetrouitPolyanskiyRigollet2025,Rigollet2025MeanFieldTransformers}.  In this note, we consider a simplified architecture called the unnormalized self-attention (USA) dynamics in the literature, which features an isotropic spherical attention interaction as we now describe.

Fix token directions $x_i\in\Sd$ and inverse temperature $\beta>0$.  In the one-head identity-parameter case, the unnormalized self-attention drift is
\begin{equation}\label{eq:intro-usa-drift}
        \dot x_i=P_{x_i}^{\perp}\left(\frac1n\sum_{j=1}^n e^{\beta x_i\cdot x_j}x_j\right),
        \qquad
        P_x^\perp=I-x\otimes x.
\end{equation}
Let $\sg$ denote normalized surface measure on $\Sd$.  The centered pair potential used throughout the paper is
\begin{equation}\label{eq:intro-centered-kernel}
\begin{gathered}
        W_\beta(x,y):=\frac{e^{\beta x\cdot y}-m_d(\beta)}{\beta},
        \qquad x,y\in\Sd,\\
        m_d(\beta):=\int_{\Sd}e^{\beta e_0\cdot z}\,d\sg(z),
        \qquad e_0\in\Sd.
\end{gathered}
\end{equation}
The quantity $m_d(\beta)$ is independent of the choice of $e_0$ by rotation invariance.  Hence, $W_\beta$ is centered, in the sense that
\begin{equation}
        \int_{\Sd}W_\beta(x,z)\,d\sg(z)=0,
        \qquad x\in\Sd.
\end{equation}
Indeed, since $m_d(\beta)$ is constant in $x$, the ambient gradient in the first variable is
\begin{equation}\label{eq:intro-gradient-potential}
        \nabla_x W_\beta(x,y)=e^{\beta x\cdot y}y.
\end{equation}
Projecting \eqref{eq:intro-gradient-potential} onto $T_x\Sd$ gives the spherical force in \eqref{eq:intro-usa-drift}.  This convention differs from the raw drift potential $(e^{\beta x\cdot y}-1)/\beta$ used in the authors' circular model \cite{MunRosenzweig2026Circle} by the additive constant $(1-m_d(\beta))/\beta$; this additive constant is immaterial for the drift and would shift the free energy by a quantity depending only on $K$ and $\beta$.

Adding isotropic Brownian noise to the spherical particle dynamics leads formally, in the mean-field limit, to a McKean--Vlasov diffusion whose stationary variational problem is the minimization of an entropy-minus-interaction free energy.  At the modeling level, this noise can be viewed as unresolved layerwise variability; at the variational level, it regularizes the attractive attention dynamics and smooths small-scale collapse mechanisms.  Following existing terminology \cite{Rigollet2025MeanFieldTransformers}, we refer to the Brownian perturbation of this spherical attention model as the ``noisy transformer.'' We now define the variational problem.

Let $\Pc(\Sd)$ be the Borel probability measures on $\Sd$, and let $\qu\equiv1$ denote the uniform density with respect to $\sg$.  We use the relative entropy $\Ent(\mu\mid\sg)$, defined precisely in \eqref{eq:entropy-definition} below.  The free energy is
\begin{equation}\label{eq:intro-free-energy}
        \F_{K,\beta}(\mu)
        :=\Ent(\mu\mid\sg)
        -K\iint_{\Sd\times\Sd}W_\beta(x,y)\,d\mu(x)d\mu(y),
        \qquad K\ge0.
\end{equation}
Here $K$ is the coupling strength, equivalently the inverse noise strength after normalization.

The phase transition problem asks when the uniform measure ceases to be a global minimizer of \eqref{eq:intro-free-energy}.  Define
\begin{equation}\label{eq:intro-Kc}
        K_c^{(d)}(\beta)
        :=\sup\left\{K\ge0:\ \{\sg\}=\argmin_{\mu\in\Pc(\Sd)}\F_{K,\beta}(\mu)\right\}.
\end{equation}
We use the terminology ``degree $\ell$'' to refer to perturbations of the uniform density in the spherical-harmonic subspace $\Hh_\ell$, and let $K_\ell^{(d)}(\beta)$ denote the coupling at which the second variation in those modes becomes neutral.  In particular, the degree-one coupling is given by
\begin{equation}\label{eq:intro-Kone}
        K_1^{(d)}(\beta)
        :=\frac{\beta^{d/2}}{2^{d/2}\Gamma(d/2)I_{d/2}(\beta)}.
\end{equation}
We let $K_{\#}^{(d)}(\beta)$ denote the threshold for the linear stability of $\qu$.  By \cref{lem:transformer-coefficients,lem:bessel-comparison} proved below, the first linear instability of $\qu$ occurs in these degree-one modes, so the linear-stability threshold  satisfies $K_\#^{(d)}(\beta)=K_1^{(d)}(\beta)$.
The transition is said to be continuous when $\sigma$ is the unique minimizer of $\F_{K,\beta}$ at $K=K_c^{(d)}(\beta)$ and supercritical global minimizers converge to $\sg$ as $K$ decreases to the critical value.  Otherwise, the transition is said to be discontinuous. 
These notions are formalized for general centered kernels in \cref{def:transition}.

This problem is part of the broader McKean--Vlasov phase-transition theory in which linear stability thresholds are compared with global minimizer thresholds \cite{ChayesPanferov2010,CarrilloGvalaniPavliotisSchlichting2020}.  For transformer-type models, the deterministic and noisy spherical dynamics have been developed by Geshkovski et al.\ in the mean-field transformer literature \cite{GeshkovskiLetrouitPolyanskiyRigollet2023Clusters,GeshkovskiLetrouitPolyanskiyRigollet2025,Rigollet2025MeanFieldTransformers}.  Layer-normalized self-attention models leading to gradient flows for probability measures on the unit sphere were analyzed by Burger et al.\ \cite{BurgerKabriKorolevRoithWeigand2025}.  In arbitrary dimension, Shalova and Schlichting studied stationary McKean--Vlasov equations on high-dimensional spheres and other compact manifolds, including transformer particle systems; they characterized critical points through the free energy, analyzed spherical-convolution bifurcations near the uniform state, and gave sufficient discontinuity criteria \cite{ShalovaSchlichting2026}.  On the circle, Balasubramanian et al.\ gave a Fourier-analytic description of stationary solutions and applied it to noisy mean-field transformers \cite{BalasubramanianBanerjeeRigollet2025}.  The authors recently proved the sharp circular variational dichotomy for the noisy-transformer kernel, with threshold $I_2(\beta_*)=I_1(\beta_*)/2$ \cite{MunRosenzweig2026Circle}.  The present note gives the corresponding sharp global-minimizer threshold in every dimension $d\ge2$.

Our main result is the following theorem.

\begin{thm}[Main result]\label{thm:main}
For every $d\ge2$, the equation
\begin{equation}\label{eq:beta-star-main}
        \frac{I_{d/2+1}(\beta_*^{(d)})}{I_{d/2}(\beta_*^{(d)})}=\frac1d
\end{equation}
has a unique solution $\beta_*^{(d)}>0$.  For the centered noisy-transformer free energy \eqref{eq:intro-free-energy}, the degree-one threshold $K_1^{(d)}(\beta)$ in \eqref{eq:intro-Kone} equals the linear stability threshold $K_\#^{(d)}(\beta)$ for every $\beta>0$, and the following statements hold.
\begin{enumerate}[(1)]
\item If $0<\beta\le\beta_*^{(d)}$, then
\begin{equation}\label{eq:main-continuous-conclusion}
        K_c^{(d)}(\beta)=K_\#^{(d)}(\beta),
\end{equation}
the transition is continuous, and $\sg$ is the unique global minimizer of $\F_{K,\beta}$ for every $0\le K\le K_\#^{(d)}(\beta)$.
\item If $\beta>\beta_*^{(d)}$, then
\begin{equation}\label{eq:main-discontinuous-conclusion}
        K_c^{(d)}(\beta)<K_\#^{(d)}(\beta),
\end{equation}
the transition is discontinuous, and global minimality of $\sigma$ is lost strictly before linear stability.
\item For every $\beta>0$ and every $K$ satisfying
\begin{equation}\label{eq:main-critical-uniqueness-range}
        0\le K\le \frac12 K_c^{(d)}(\beta),
\end{equation}
the uniform density $\qu$ is the unique critical point of $\F_{K,\beta}$.  In particular, in the continuous regime $0<\beta\le\beta_*^{(d)}$, the uniform density is the unique critical point for $0\le K\le K_\#^{(d)}(\beta)/2$.
\end{enumerate}
\end{thm}

\begin{remark}[Small-$\beta$ limit and Kuramoto model]
The small-$\beta$ limit connects the spherical attention interaction with the classical synchronization model on the circle.  Since $m_d(\beta)=1+O(\beta^2)$, the centered kernel satisfies $W_\beta(x,y)=x\cdot y+O(\beta)$, while the particle drift in \eqref{eq:intro-usa-drift} converges to $P_{x_i}^{\perp}n^{-1}\sum_{j=1}^n x_j$.  In dimension $d=2$, writing $x_i=(\cos\theta_i,\sin\theta_i)$ gives
\begin{equation}
       \frac1n\sum_{j=1}^n e^{\beta\cos(\theta_i-\theta_j)}
        \sin(\theta_j-\theta_i)
        \longrightarrow
        \frac1n\sum_{j=1}^n\sin(\theta_j-\theta_i),
\end{equation}
which is the drift in the identical-oscillator Kuramoto model without frequency disorder \cite{Kuramoto1975}.  The Kuramoto model and its noisy variants have an extensive literature, and much is known in both the physics and mathematics literatures, a proper review of which is beyond the scope of this paper.  We refer the interested reader to \cite{Strogatz2000,AcebronBonillaPerezVicenteRitortSpigler2005}, \cite[Section 3.1]{GeshkovskiLetrouitPolyanskiyRigollet2025}, \cite[Section 1.2]{MunRosenzweig2026KuramotoDaido} and the references therein.
\end{remark}

The discontinuous alternative in \cref{thm:main} identifies the strict separation between the global-minimizer threshold and the linear-stability threshold.  The exact threshold can equivalently be written as a quotient.  If
\begin{equation}\label{eq:intro-Ebeta-quotient}
        E_\beta(q)=\iint_{\Sd\times\Sd}e^{\beta x\cdot y}q(x)q(y)\,d\sg(x)d\sg(y),
\end{equation}
denotes the interaction energy, then, in the normalization of \eqref{eq:intro-free-energy},
\begin{equation}\label{eq:intro-Kc-quotient}
        K_c^{(d)}(\beta)
        =\inf_{\substack{q\ge0,\ \int_{\Sd}q\,d\sg=1,\\ E_\beta(q)-E_\beta(1)>0}}
        \frac{\beta\Ent(q\sg\mid\sg)}{E_\beta(q)-E_\beta(1)}.
\end{equation}
Thus, the theorem determines when this quotient is equal to $K_\#^{(d)}(\beta)$, but it does not give a closed form for the infimum when $\beta>\beta_*^{(d)}$; the formal asymptotic discussion below addresses this threshold in two limiting regimes.

For reference, numerical solutions of \eqref{eq:beta-star-main} in low dimensions are
\begin{equation}\label{eq:beta-star-table}
\begin{gathered}
\begin{array}{c|ccccc}
 d&2&3&4&5&6\\ \hline
 \beta_*^{(d)}&2.446918&1.816155&1.575257&1.445485&1.363867
\end{array}\\[0.4em]
\begin{array}{c|cccc}
 d&7&8&9&10\\ \hline
 \beta_*^{(d)}&1.307662&1.266553&1.235161&1.210397
\end{array}
\end{gathered}
\end{equation}
The function $d\mapsto\beta_*^{(d)}$ is strictly decreasing (see \cref{lem:bessel-comparison}).
The large-$d$ expansion proved in \cref{lem:bessel-comparison} is
\begin{equation}\label{eq:intro-beta-star-large-d}
        \beta_*^{(d)}=1+\frac2d+\frac1{d^2}+O(d^{-3}).
\end{equation}

We also record formal asymptotic predictions for the exact discontinuous threshold. These formulas are not used in the proof of \cref{thm:main}; they are obtained by a finite-dimensional Lyapunov--Schmidt/Landau expansion of the quotient \eqref{eq:intro-Kc-quotient} near $\beta_*^{(d)}$ and by a localized-cap ansatz as $\beta\to\infty$.  Let $r_\ell(\beta)$ be as in \eqref{eq:intro-ratio-formula} below.  As $\beta\downarrow\beta_*^{(d)}$, one expects
\begin{equation}\label{eq:intro-tricritical-Kc-preview}
        K_\#^{(d)}(\beta)-K_c^{(d)}(\beta)
        =c_{\mathrm{tri}}^{(d)}
        (\beta-\beta_*^{(d)})^2
        +o\bigl((\beta-\beta_*^{(d)})^2\bigr),
\end{equation}
where $c_{\mathrm{tri}}^{(d)}>0$ is the constant produced by the formal finite-dimensional Landau expansion.
At the opposite end, if $|\Sd|$ denotes unnormalized surface area, the predicted large-$\beta$ behavior in the present normalization is
\begin{equation}\label{eq:intro-large-beta-Kc-preview}
        K_c^{(d)}(\beta)
        =\frac{\beta e^{-\beta}}2
        \left[(d-1)\log\beta+(d-1)\log\log\beta+C_\infty(d)+o(1)\right],
        \qquad \beta\to\infty,
\end{equation}
where
\begin{equation}\label{eq:intro-large-beta-constant-preview}
        C_\infty(d)=(d-1)\log(d-1)+2\log|\Sd|-(d-1)\log(2\pi).
\end{equation}
The rigorous verification of these asymptotics, and of the structure of minimizers for the quotient \eqref{eq:intro-Kc-quotient}, is the subject of ongoing work by the authors.

The proof of \cref{thm:main} is the higher-dimensional analogue of the authors' circular argument \cite{MunRosenzweig2026Circle}, with the harmonic-analysis input changed.  On the circle, the coercive entropy estimate comes from the classical Lebedev--Milin inequality \cite{LebedevMilin1965}.  On $\Sd$, it comes from the sharp Beckner--Onofri inequality \cite{Beckner1993}, equivalently the logarithmic Hardy-Littlewood-Sobolev (HLS) inequality of Carlen and Loss \cite{CarlenLoss1992}.  Its dual form controls the entropy by spherical harmonic projections with weights
\begin{equation}\label{eq:intro-beckner-weights}
        b_{\ell,d}=\frac{\Gamma(\ell)\Gamma(d)}{\Gamma(\ell+d-1)}
        =\prod_{j=1}^{\ell-1}\frac{j}{d+j-1},
        \qquad \ell\ge1.
\end{equation}
The Funk--Hecke coefficients of $W_\beta$ are modified Bessel functions as computed in \cref{lem:transformer-coefficients}, and their ratios are
\begin{equation}\label{eq:intro-ratio-formula}
        r_\ell(\beta)=\frac{I_{\ell+d/2-1}(\beta)}{I_{d/2}(\beta)},
        \qquad \ell\ge1.
\end{equation}
Segura's monotonicity theorem for Bessel ratios \cite{Segura2021} implies $r_\ell(\beta)\le b_{\ell,d}$ for all $\ell\ge2$ exactly in the parameter range needed for the continuous side of \cref{thm:main}.  When $r_2(\beta)>b_{2,d}=1/d$, a perturbation mixing a neutral degree-one mode with the degree-two component generated by its square lowers the free energy at order four.  This is the degree-two quartic obstruction proving the discontinuous side.

The rest of the paper is organized as follows.  \Cref{sec:formulation} isolates variational facts for general centered zonal kernels, including critical-point uniqueness below $K_c/2$.  \Cref{sec:beckner} proves the Beckner--Onofri coercivity criterion.  \Cref{sec:transformer-coefficients} computes the noisy-transformer Funk--Hecke coefficients and proves the Bessel-ratio comparison.  \Cref{sec:quartic} gives the degree-two quartic obstruction.  \Cref{sec:proof-main} assembles these ingredients to prove \cref{thm:main}.

\medskip
\noindent\textbf{Acknowledgement.} The second author thanks Borjan Geshkovski for originally sparking his interest in the mathematics of transformers.
\medskip

\section{Variational formulation}\label{sec:formulation}

We first separate the abstract variational mechanism from the special form of the attention kernel.  Throughout this section $W$ is a continuous centered zonal kernel on the sphere.  The noisy-transformer kernel will be inserted only after the general endpoint criteria are established.

We use normalized surface measure $\sg$ throughout.  For background on spherical harmonics, see \cite{AtkinsonHan2012}.  The spherical harmonic decomposition is
\begin{equation}\label{eq:spherical-harmonic-decomposition}
        L^2(\Sd,\sg)=\bigoplus_{\ell=0}^\infty \Hh_\ell,
\end{equation}
where $\Hh_\ell$ is the space of degree-$\ell$ spherical harmonics.  We write $\Pi_\ell$ for the $L^2(\sg)$-orthogonal projection onto $\Hh_\ell$.  Let $\Delta_{\Sd}$ denote the Laplace--Beltrami operator on $\Sd$.  With the sign convention $-\Delta_{\Sd}Y=a_{\ell,d}Y$ for $Y\in\Hh_\ell$, the corresponding positive eigenvalue is
\begin{equation}\label{eq:laplace-eigenvalue}
        a_{\ell,d}=\ell(\ell+d-2).
\end{equation}
The reproducing kernel of $\Hh_\ell$ is denoted by $Z_{\ell,d}$ and normalized by
\begin{equation}\label{eq:reproducing-kernel}
        \Pi_\ell f(x)=\int_{\Sd}Z_{\ell,d}(x,y)f(y)\,d\sg(y).
\end{equation}
Since each $\Hh_\ell$ is finite-dimensional and consists of smooth functions, $\Pi_\ell u$ is well-defined as an element of $\Hh_\ell$ for any distribution $u$ on $\Sd$, by duality against $\Hh_\ell$.  Thus, expressions such as $\Pi_\ell(q-\qu)$ are understood distributionally whenever $q-\qu$ is not a priori in $L^2(\sg)$.

As mentioned in the introduction, our notation for the relative entropy is
\begin{equation}\label{eq:entropy-definition}
        \Ent(\mu\mid\sg):=
        \begin{cases}
        \displaystyle\int_{\Sd}q\log q\,d\sg, & \mu=q\sg,\ q\log q\in L^1(\sg),\\
        +\infty, & \text{otherwise.}
        \end{cases}
\end{equation}
In particular, finite free energy forces absolute continuity.

Let $W:\Sd\times\Sd\to\R$ be continuous, symmetric, centered, and zonal.  Centered means
\begin{equation}\label{eq:centering}
        T_W\qu=0,
        \qquad
        T_Wf(x)=\int_{\Sd}W(x,y)f(y)\,d\sg(y).
\end{equation}
Zonal means that $W(x,y)$ is a function of $x\cdot y$.  Then the Funk--Hecke formula diagonalizes $T_W$ on spherical harmonics \cite[Section 2.6]{AtkinsonHan2012}: there are real eigenvalues $\lambda_\ell$ such that
\begin{equation}\label{eq:funk-hecke-general}
        T_WY=\lambda_\ell Y,
        \qquad Y\in\Hh_\ell,
        \qquad \lambda_0=0.
\end{equation}
The associated free energy is
\begin{equation}\label{eq:general-free-energy}
        \F_K(\mu)=\Ent(\mu\mid\sg)-K\iint_{\Sd\times\Sd}W(x,y)\,d\mu(x)d\mu(y).
\end{equation}
If $\mu=q\sg$ and $f=q-\qu$, centering gives the gap identity
\begin{equation}\label{eq:gap-identity-formal}
        \F_K(q\sg)-\F_K(\sg)=\Ent(q\sg\mid\sg)-K\ip{f}{T_Wf}.
\end{equation}

For $q_\ep=\qu+\ep f$ with $f\in\Hh_\ell$, $\norm{f}_2=1$, and $|\ep|$ small, the $\ep^2$ coefficient in $\F_K(q_\ep\sg)-\F_K(\sg)$ is $\frac12(1-2K\lambda_\ell)$; equivalently, the normalized second variation in that direction is $1-2K\lambda_\ell$.  Hence, the first linear stability threshold is
\begin{equation}\label{eq:Ksharp-general}
        K_\#=\frac{1}{2\lambda_*},
        \qquad
        \lambda_*:=\max_{\ell\ge1}\max\{\lambda_\ell,0\}.
\end{equation}
Assuming $\lambda_1\ge0$ and letting $K_1 = \frac{1}{2\lambda_1}$ denote the threshold for degree-one stability, the model is degree-one-leading when $K_\#=K_1$. Equivalently, defining the normalized eigenvalue ratios
\begin{equation}\label{eq:Kone-ratios-general}
        r_\ell:=\frac{\lambda_\ell}{\lambda_1},\qquad \ell\ge1,
\end{equation}
the model is degree-one-leading when $r_\ell\le1$ for every $\ell\ge2$.

\begin{mydef}[Critical coupling and transition type]\label{def:transition}
For a centered kernel $W$ satisfying the preceding assumptions, let
\begin{equation}\label{eq:MK-Kc}
        \M_K=\argmin_{\mu\in\Pc(\Sd)}\F_K(\mu),
        \qquad
        K_c=\sup\{K\ge0:\ \{\sg\}=\M_K\}.
\end{equation}
The transition is called \emph{continuous} if  $\mathcal{M}_{K_c}=\{\sg\}$ and, whenever $K_n\downarrow K_c$ and $\mu_n\in\M_{K_n}$, every weak limit point of $(\mu_n)$ equals $\sg$.  Otherwise, the transition is called \emph{discontinuous}.  
\end{mydef}

The following elementary lemma explains why the endpoint $K_\#$ is decisive.  If the uniform measure is the unique minimizer at the linear threshold, then the transition is continuous.  If a competitor already beats the uniform measure at that threshold, then global minimality must have been lost earlier.  This is the spherical analogue, in the present notation, of the transition-point criterion in \cite[Proposition 2.12]{ChayesPanferov2010}, \cite[Proposition~5.8]{CarrilloGvalaniPavliotisSchlichting2020}.

\begin{lemma}[Basic variational implications]\label{lem:Kc}
Let $W$ be continuous, symmetric, and centered.  Then $\F_K$ has a minimizer in $\Pc(\Sd)$ for every $K\ge0$.  Moreover:
\begin{enumerate}[(i)]
\item if $\sg$ (uniquely) minimizes $\F_K$, then $\sg$ (uniquely) minimizes $\F_{K'}$ for every $0\le K'\le K$;
\item if $\sg$ is the unique minimizer of $\F_{K_\#}$, then $K_c=K_\#$ and the transition is continuous;
\item if there exists $\mu\in\Pc(\Sd)$ such that $\F_{K_\#}(\mu)<\F_{K_\#}(\sg)$, then $K_c<K_\#$ and the transition is discontinuous.
\end{enumerate}
\end{lemma}

\begin{proof}
Existence follows from weak compactness of $\Pc(\Sd)$, lower semicontinuity of entropy, and continuity of the interaction energy under weak convergence.  Put
\begin{equation}\label{eq:centered-energy}
        \E(\mu)=\iint W\,d\mu d\mu-\iint W\,d\sg d\sg.
\end{equation}
By centering, the second term vanishes.  If $\sg$ minimizes $\F_K$, then
\begin{equation}\label{eq:minimizer-inequality-at-K}
        \Ent(\mu\mid\sg)-K\E(\mu)\ge0
        \qquad\text{for every }\mu\in\Pc(\Sd).
\end{equation}
If $\E(\mu)>0$, decreasing $K$ preserves \eqref{eq:minimizer-inequality-at-K}.  If $\E(\mu)\le0$, the inequality for smaller $K$ follows from nonnegativity of entropy.  The same argument also gives downward propagation of uniqueness. Suppose that
\(\sigma\) is the unique minimizer of \(\mathcal F_K\), let \(0\le K'<K\), and let
\(\mu\) minimize \(\mathcal F_{K'}\). Then since $\sigma$ also minimizes $\F_{K'}$, we have
\begin{align}
        H(\mu\mid\sigma)-K'\mathcal E(\mu)=0,
        \qquad
        H(\mu\mid\sigma)-K\mathcal E(\mu)\ge0 .
\end{align}
Subtracting gives \((K-K')\mathcal E(\mu)\le0\), hence \(\mathcal E(\mu)\le0\). The first
equality then gives \(H(\mu\mid\sigma)\le0\), so \(H(\mu\mid\sigma)=0\) and therefore
\(\mu=\sigma\). This proves (i).

For (ii), uniqueness at $K_\#$ gives $K_c\ge K_\#$.  If $K>K_\#$, the second variation is negative in a spherical-harmonic degree attaining $\lambda_*$, so $\sg$ is not a local minimizer and hence $K_c\le K_\#$.  Let $K_n\downarrow K_c$ and $\mu_n\in\M_{K_n}$.  Since
\begin{equation}\label{eq:limit-minimizer-ineq}
        \F_{K_n}(\mu_n)\le\F_{K_n}(\sg)=0,
\end{equation}
while $K_n\to K_c$ and the interaction energy is continuous under weak convergence, every weak limit point $\mu$ of $(\mu_n)$ satisfies $\F_{K_c}(\mu)\le0$.  Because $K_c=K_\#$ and $\sg$ is the unique minimizer at $K_\#$, the limit point is $\sg$.

For (iii), strict inequality at $K_\#$ implies $\E(\mu)>0$, since otherwise the gap would be at least $\Ent(\mu\mid\sg)\ge0$.  Therefore, $K\mapsto \F_K(\mu)-\F_K(\sg)$ remains negative for all $K<K_\#$ sufficiently close to $K_\#$, and $\sg$ is not minimizing for such $K$, implying $K_c<K_\#$. If \(K_c<K_\#\), then either \(\mathcal M_{K_c}\ne\{\sigma\}\), in which case the transition is discontinuous by definition, or \(\mathcal M_{K_c}=\{\sigma\}\). In the latter case, supercritical minimizers cannot converge weakly to \(\sigma\). Indeed, choosing \(\bar K\in(K_c,K_\#)\), the Hessian at \(\sigma\) is uniformly positive for \(K\le \bar K\). Suppose that there is a sequence $\mu_n\to\sigma$ weakly, where each $\mu_n$ is a minimizer of $\F_{K_n}$ for some $K_n\to K_c$. The Kirkwood--Monroe equation upgrades weak convergence of minimizers to uniform convergence of their densities to \(1\), contradicting the strict local stability unless the $\mu_n$ are identically \(\sigma\). Hence, the convergence condition in \Cref{def:transition} fails.
\end{proof}

We next record the critical-point uniqueness mechanism.  This is the generalization of the $K_c/2$ uniqueness statement in the circular model; the proof proceeds by the same argument, mutatis mutandis.  It uses only the variational threshold and the Kirkwood--Monroe equation, not the specific harmonic analysis of the transformer kernel.

A critical point will mean a probability density solving the Euler--Lagrange, or Kirkwood--Monroe, equation associated with \eqref{eq:general-free-energy}.  More precisely, $q\sg$ is a critical point at coupling $K$ if
\begin{equation}\label{eq:KM-general}
        q(x)=\frac{\exp\{2K T_W(q-\qu)(x)\}}{\int_{\Sd}\exp\{2K T_W(q-\qu)(z)\}\,d\sg(z)}.
\end{equation}
Because $T_W\qu=0$ by centering, this is equivalent to writing the equation with $T_Wq$; the form above keeps the centered perturbation visible.
Since $W$ is continuous on the compact space $\Sd\times\Sd$, it is bounded; hence every solution of \eqref{eq:KM-general} is bounded above and below by positive constants.

\begin{prop}[Critical-point uniqueness]\label{prop:critical-uniqueness-half}
Let $W$ be continuous, symmetric, centered, and zonal, and assume that $K_c<\infty$.  If $q\sg$ is a critical point of $\F_K$ and
\begin{equation}\label{eq:half-critical-range}
        0\le K\le \frac12K_c,
\end{equation}
then $q=\qu$.
\end{prop}

\begin{proof}
For $K=0$, \eqref{eq:KM-general} immediately gives $q=\qu$.  Assume $K>0$.  Because $K_c$ is finite and the set of $K$ for which $\sg$ minimizes is an interval by \cref{lem:Kc}(i), passing to the limit along $K_n\uparrow K_c$ with $\sg\in\M_{K_n}$ gives
\begin{equation}\label{eq:sigma-minimizes-at-Kc}
        \F_{K_c}(\mu)-\F_{K_c}(\sg)\ge0
        \qquad\text{for every }\mu\in\Pc(\Sd).
\end{equation}
Let $f=q-\qu$.  From \eqref{eq:KM-general}, we have
\begin{equation}\label{eq:log-critical-density}
        \log q=2K T_Wf-\log Z,
        \qquad
        Z=\int_{\Sd}e^{2K T_Wf}\,d\sg.
\end{equation}
Using centering, we obtain the entropy identity
\begin{align}
        \Ent(q\sg\mid\sg)+\Ent(\sg\mid q\sg)
        &=\int_{\Sd}q\log q\,d\sg-\int_{\Sd}\log q\,d\sg \notag\\
        &=2K\ip{f}{T_Wf}.
        \label{eq:critical-entropy-identity}
\end{align}
The second entropy in \eqref{eq:critical-entropy-identity} is finite because $q$ is bounded below.  Since $\sg$ minimizes at $K_c$, \eqref{eq:sigma-minimizes-at-Kc} applied to $q\sg$ gives
\begin{equation}\label{eq:Kc-energy-entropy-bound}
        K_c\ip{f}{T_Wf}\le \Ent(q\sg\mid\sg).
\end{equation}
The left-hand side of \eqref{eq:critical-entropy-identity} is nonnegative, so $\ip{f}{T_Wf}\ge0$.  If $K\le K_c/2$, then \eqref{eq:critical-entropy-identity} and \eqref{eq:Kc-energy-entropy-bound} imply
\begin{equation}\label{eq:half-critical-final-bound}
        \Ent(q\sg\mid\sg)+\Ent(\sg\mid q\sg)
        =2K\ip{f}{T_Wf}
        \le K_c\ip{f}{T_Wf}
        \le \Ent(q\sg\mid\sg).
\end{equation}
Thus, $\Ent(\sg\mid q\sg)\le0$.  Since relative entropy is nonnegative, $\Ent(\sg\mid q\sg)=0$, and therefore $q=\qu$.
\end{proof}

\begin{remark}[Open uniqueness range]
\Cref{prop:critical-uniqueness-half} leaves open the range
\begin{equation}\label{eq:half-to-full-range}
        \frac12K_c<K\le K_c.
\end{equation}
We tentatively conjecture that, whenever the transition is continuous in the sense of \cref{def:transition}, $\sigma$ is the unique critical point in this omitted range.
\end{remark}

\section{The Beckner--Onofri coercivity mechanism}\label{sec:beckner}

We now prove the coercive entropy estimate that replaces the Lebedev--Milin inequality in the circular proof.  The proof uses two preliminary inputs: the dual Beckner--Onofri/logarithmic HLS inequality and the equality-case information needed for endpoint uniqueness. 

The constants used below are the Beckner weights, namely the harmonic weights appearing in the dual Beckner--Onofri/logarithmic HLS inequality.
Define
\begin{equation}\label{eq:beckner-weights}
        b_{\ell,d}=\frac{\Gamma(\ell)\Gamma(d)}{\Gamma(\ell+d-1)}
        =\prod_{j=1}^{\ell-1}\frac{j}{d+j-1},
        \qquad \ell\ge1,
\end{equation}
where the empty product is, by convention, one.  In particular,
\begin{equation}\label{eq:first-beckner-weights}
        b_{1,d}=1,
        \qquad
        b_{2,d}=\frac1d,
        \qquad
        b_{3,d}=\frac{2}{d(d+1)}.
\end{equation}
The sharp Beckner--Onofri inequality on $\Sd$ \cite{Beckner1993}, equivalently the sharp logarithmic HLS inequality and equality-case classification of Carlen--Loss \cite{CarlenLoss1992}, gives the following dual form.

\begin{prop}[Dual Beckner--Onofri/logarithmic HLS inequality]\label{prop:dualBO}
Let $\mu=q\sg$ have finite entropy.  Then
\begin{equation}\label{eq:dualBO}
        \Ent(q\sg\mid\sg)
        \ge \frac12\sum_{\ell=1}^\infty b_{\ell,d}\norm{\Pi_\ell(q-\qu)}_2^2,
\end{equation}
where the projections are understood distributionally.  Equality holds if and only if
\begin{equation}\label{eq:conformal-jacobian}
        q(x)=J_a(x):=\left(\frac{1-|a|^2}{|x-a|^2}\right)^{d-1},
        \qquad a\in\R^d,
        \qquad |a|<1.
\end{equation}
In particular, $J_0\equiv1$.
\end{prop}

\begin{proof}
For every bounded measurable test function $\phi$, the Donsker--Varadhan variational formula for relative entropy \cite{DonskerVaradhan1975} gives
\begin{equation}\label{eq:dv}
        \Ent(q\sg\mid\sg)
        \ge \int_{\Sd}\phi(q-\qu)\,d\sg
        -\left(\log\int_{\Sd}e^\phi\,d\sg-\int_{\Sd}\phi\,d\sg\right).
\end{equation}
The primal Beckner--Onofri inequality states that
\begin{equation}\label{eq:primalBO}
        \log\int_{\Sd}e^\phi\,d\sg-\int_{\Sd}\phi\,d\sg
        \le \frac12\sum_{\ell\ge1} b_{\ell,d}^{-1}\norm{\Pi_\ell\phi}_2^2
\end{equation}
for smooth $\phi$, with the usual extension by approximation.  
Apply \eqref{eq:primalBO} to a finite harmonic polynomial $\phi_N=\sum_{\ell=1}^N\phi_\ell$, $\phi_\ell\in\Hh_\ell$.  Combining \eqref{eq:dv} and \eqref{eq:primalBO} gives
\begin{align}
        \Ent(q\sg\mid\sg)
        \ge{}&
        \sum_{\ell=1}^N\ip{\phi_\ell}{\Pi_\ell(q-\qu)}
        -\frac12\sum_{\ell=1}^N b_{\ell,d}^{-1}\norm{\phi_\ell}_2^2.
\end{align}
The right-hand side is a finite-dimensional concave quadratic expression.  Optimizing independently in each $\phi_\ell$ gives the maximizer $\phi_\ell=b_{\ell,d}\Pi_\ell(q-\qu)$ and contributes $\frac12 b_{\ell,d}\norm{\Pi_\ell(q-\qu)}_2^2$ in degree $\ell$.  Thus,
\begin{equation}\label{eq:finite-dualBO}
        \Ent(q\sg\mid\sg)
        \ge \frac12\sum_{\ell=1}^N b_{\ell,d}\norm{\Pi_\ell(q-\qu)}_2^2.
\end{equation}
Letting $N\to\infty$ proves \eqref{eq:dualBO}.  The stated equality cases are precisely the equality cases in the sharp Beckner--Onofri/logarithmic HLS theorem.
\end{proof}

\begin{remark}\label{rem:lebedev-milin}
For $d=2$, the conformal Jacobians in \eqref{eq:conformal-jacobian} reduce to Poisson kernels on the circle, and the weights are $b_{\ell,2}=\frac1\ell$. 
This is the same spectral scale as the dual entropy estimate in \cite[Proposition~2.2]{MunRosenzweig2026Circle}, derived there from the Lebedev--Milin inequality.  In dimensions $d\ge3$, \cref{prop:dualBO} is the conformally sharp spherical substitute: Fourier modes are replaced by spherical harmonics, and the coefficients $1/\ell$ are replaced by \eqref{eq:beckner-weights}.
\end{remark}

For endpoint uniqueness, equality in \cref{prop:dualBO} leaves the conformal Jacobian densities; later strictness in a particular harmonic degree rules out the nonconstant ones only if every such density has a nonzero projection in that degree.
\begin{lemma}\label{lem:nonzero-conformal}
If $a\ne0$ in \eqref{eq:conformal-jacobian}, then
\begin{equation}\label{eq:nonzero-projections}
        \Pi_\ell(J_a-\qu)\ne0
        \qquad\text{for every }\ell\ge1.
\end{equation}
\end{lemma}

\begin{proof}
By rotation invariance, take $a=re_d$, where $e_d$ is the last coordinate vector in $\R^d$, with $0<r<1$, and put $t=x\cdot e_d$.

If $d=2$, writing $x=e^{i\theta}$ gives the Poisson-kernel expansion
\begin{equation}\label{eq:poisson-expansion}
        J_a(e^{i\theta})=\frac{1-r^2}{1-2r\cos\theta+r^2}
        =1+2\sum_{\ell=1}^\infty r^\ell\cos(\ell\theta),
\end{equation}
so every nonzero Fourier mode appears.

Assume now that $d\ge3$ and set $\alpha=(d-2)/2$.  We use the standard Gegenbauer/ultraspherical normalization of \cite[Section~18.5]{NISTDLMF}.  The zonal harmonic of degree $\ell$ is represented by the Gegenbauer polynomial $C_\ell^\alpha$.  For this normalization, Rodrigues' formula \cite[Table~18.5.1]{NISTDLMF} gives, up to a positive $\ell,d$-dependent constant,
\begin{equation}\label{eq:gegenbauer-rodrigues}
        (1-t^2)^{\alpha-1/2}C_\ell^\alpha(t)
        =(-1)^\ell\frac{d^\ell}{dt^\ell}
        \left[(1-t^2)^{\ell+\alpha-1/2}\right].
\end{equation}
For
\begin{equation}\label{eq:conformal-zonal-profile}
        F_r(t)=\left(\frac{1-r^2}{1-2rt+r^2}\right)^{d-1},
\end{equation}
integrating by parts $\ell$ times shows that the degree-$\ell$ zonal coefficient of $F_r$ is a positive constant times
\begin{equation}\label{eq:positive-zonal-coeff}
        \int_{-1}^{1}
        (2r)^\ell(d-1)_\ell(1-2rt+r^2)^{-(d-1)-\ell}
        (1-t^2)^{\ell+\alpha-1/2}\,dt.
\end{equation}
The integral in \eqref{eq:positive-zonal-coeff} is strictly positive.  Hence, the projection onto $\Hh_\ell$ is nonzero for every $\ell\ge1$.
\end{proof}

The Beckner--Onofri inequality controls harmonic projections of finite-entropy densities, which need not belong to $L^2$.  We now combine it with the heat-kernel regularization of the interaction to prove the continuity criterion that will be applied to the noisy-transformer eigenvalues.

\begin{remark}\label{rem:ordinary-spectral-convergence}
For a finite-entropy density $q$, the heat-regularized formulation in the proof below may be replaced by ordinary spectral sums whenever
\begin{equation}
        \sum_{\ell=1}^\infty |\lambda_\ell|\norm{\Pi_\ell(q-\qu)}_2^2<\infty.
\end{equation}
One sufficient condition is that $W$ be positive semidefinite on the centered subspace, equivalently $\lambda_\ell\ge0$ for every $\ell\ge1$.  In that case, monotone convergence applied to the heat-regularized identity gives
\begin{equation}\label{eq:positive-definite-ordinary-series}
        \iint_{\Sd\times\Sd} W(x,y)(q-\qu)(x)(q-\qu)(y)\,d\sg(x)d\sg(y)
        =\sum_{\ell=1}^\infty \lambda_\ell
        \norm{\Pi_\ell(q-\qu)}_2^2<\infty.
\end{equation}
More generally, the same conclusion holds if $W=W_+-W_-$, where $W_+$ and $W_-$ are continuous, symmetric, zonal, positive-semidefinite kernels.  Their degree-zero modes are irrelevant in the centered pairing with $q-\qu$.  If $\lambda_\ell^\pm\ge0$ denote their Funk--Hecke eigenvalues, then applying \eqref{eq:positive-definite-ordinary-series} separately to $W_+$ and $W_-$ yields
\begin{equation}\label{eq:absolute-convergence-positive-negative-parts}
        \sum_{\ell=1}^\infty |\lambda_\ell|
        \norm{\Pi_\ell(q-\qu)}_2^2
        \le
        \sum_{\ell=1}^\infty (\lambda_\ell^+ + \lambda_\ell^-)
        \norm{\Pi_\ell(q-\qu)}_2^2
        <\infty.
\end{equation}
Thus, in this situation, all heat-regularized interaction sums in the proof can be read as ordinary absolutely convergent series.  This additional hypothesis is not automatic for a general continuous sign-indefinite kernel: its positive and negative spectral parts need not themselves be represented by continuous kernels.
\end{remark}

\begin{thm}[Beckner--Onofri continuity criterion]\label{thm:BOcriterion}
Let $W$ be continuous, symmetric, centered, and zonal.  Assume $\lambda_1>0$, {let $K_1$, $K_\#$, and $r_\ell$ be as in} \eqref{eq:Kone-ratios-general} {and} \eqref{eq:Ksharp-general}, and assume that {$K_\#=K_1$}.  If
\begin{equation}\label{eq:BOcriterion}
        r_\ell\le b_{\ell,d}
        \qquad {\forall}\,\ell\ge2,
\end{equation}
with strict inequality for at least one degree $\ell\ge2$, then $K_c=K_\#$, the transition is continuous, and $\sg$ is the unique global minimizer of $\F_K$ for every $0\le K\le K_\#$.
\end{thm}

\begin{proof}
It suffices to consider probability densities $q$ with finite entropy, since otherwise the free-energy gap is $+\infty$.  Let $P_t=e^{t\Delta_{\Sd}}$ be the spherical heat semigroup (see \cite[Chapters~2--3]{AtkinsonHan2012} for the spherical-harmonic and Laplace--Beltrami background behind the heat-semigroup properties used here).  Set
\begin{equation}\label{eq:heat-regularized-kernel}
        W_t=P_t^xP_t^yW.
\end{equation}
Then $W_t$ is smooth, centered, and zonal, with Funk--Hecke eigenvalues $e^{-2ta_{\ell,d}}\lambda_\ell$, and $W_t\to W$ uniformly as $t\downarrow0$.  Moreover, $P_t(q-\qu)$ is smooth and $\Pi_\ell P_t(q-\qu)=e^{-ta_{\ell,d}}\Pi_\ell(q-\qu)$ for every $t>0$, and
\begin{equation}\label{eq:heat-pushed-interaction}
        \iint_{\Sd\times\Sd}W_t(x,y)(q-\qu)(x)(q-\qu)(y)\,d\sg(x)d\sg(y)
        =\ip{P_t(q-\qu)}{T_WP_t(q-\qu)}.
\end{equation}
The usual $L^2$ spectral identity therefore gives
\begin{align}
        &\iint_{\Sd\times\Sd} W_t(x,y)(q-\qu)(x)(q-\qu)(y)\,d\sg(x)d\sg(y) \notag\\
        &\qquad =
        \sum_{\ell=1}^\infty e^{-2ta_{\ell,d}}\lambda_\ell
        \norm{\Pi_\ell(q-\qu)}_2^2.
        \label{eq:heat-spectral-identity}
\end{align}
Equivalently, the right-hand side is the heat-kernel, or Abel, regularization of the formal interaction series.  Passing $t\downarrow0$ in the left-hand side is justified by the uniform convergence $W_t\to W$ and gives
\begin{align}
        &\iint_{\Sd\times\Sd} W(x,y)(q-\qu)(x)(q-\qu)(y)\,d\sg(x)d\sg(y) \notag\\
        &\qquad =
        \lim_{t\downarrow0}
        \sum_{\ell=1}^\infty e^{-2ta_{\ell,d}}\lambda_\ell
        \norm{\Pi_\ell(q-\qu)}_2^2.
        \label{eq:heat-spectral-limit}
\end{align}
Combining \eqref{eq:heat-spectral-limit} with \cref{prop:dualBO}, we obtain, for every finite $K$,
\begin{align}
        \F_K(q\sg)-\F_K(\sg)
        ={}&
        \left[
        \Ent(q\sg\mid\sg)
        -\frac12\sum_{\ell=1}^\infty
        b_{\ell,d}\norm{\Pi_\ell(q-\qu)}_2^2
        \right] \notag\\
        &+\frac12\lim_{t\downarrow0}
        \sum_{\ell=1}^\infty
        \left(b_{\ell,d}-2K e^{-2ta_{\ell,d}}\lambda_\ell\right)
        \norm{\Pi_\ell(q-\qu)}_2^2.
        \label{eq:heat-gap-decomposition}
\end{align}
The first term on the right-hand side is finite by \cref{prop:dualBO}. The limit in the second term exists because it is equal to
\begin{equation}\label{eq:finite-heat-limit}
        \sum_{\ell=1}^\infty b_{\ell,d}\norm{\Pi_\ell(q-\qu)}_2^2
        -2K
        \iint_{\Sd\times\Sd} W(x,y)(q-\qu)(x)(q-\qu)(y)\,d\sg(x)d\sg(y),
\end{equation}
and the interaction integral is finite because $W$ is bounded and $q-\qu\in L^1(\sg)$.

We next verify the coefficient comparison.  {If $K\le K_\#=K_1$, then}
\begin{equation}\label{eq:spectral-majorization}
        K\lambda_\ell\le\frac12 b_{\ell,d}
        \qquad\forall \ell\ge1.
\end{equation}
Indeed, in degree one,
\begin{equation}\label{eq:degree-one-majorization}
        K\lambda_1\le {K_1}\lambda_1=\frac12=\frac12 b_{1,d}.
\end{equation}
For $\ell\ge2$, split according to the sign of $\lambda_\ell$.  If $\lambda_\ell\ge0$, then
\begin{equation}\label{eq:higher-degree-majorization-positive}
        K\lambda_\ell\le {K_1}\lambda_\ell=\frac12 r_\ell\le\frac12b_{\ell,d},
\end{equation}
where the final inequality is by our assumption \eqref{eq:BOcriterion}. If $\lambda_\ell<0$, then trivially
\begin{equation}\label{eq:higher-degree-majorization-negative}
        K\lambda_\ell\le 0\le\frac12b_{\ell,d}.
\end{equation}

Under \eqref{eq:spectral-majorization}, the last term in \eqref{eq:heat-gap-decomposition} is an ordinary nonnegative series.  Indeed,
\begin{equation}\label{eq:regularized-coefficient-split}
        b_{\ell,d}-2K e^{-2ta_{\ell,d}}\lambda_\ell
        =e^{-2ta_{\ell,d}}\left(b_{\ell,d}-2K\lambda_\ell\right)
        +\left(1-e^{-2ta_{\ell,d}}\right)b_{\ell,d}.
\end{equation}
As $t\downarrow0$, the first term on the right-hand side of \eqref{eq:regularized-coefficient-split} increases monotonically to $b_{\ell,d}-2K\lambda_\ell$, while the second term tends to zero and is dominated by $b_{\ell,d}$.  Therefore, by monotone convergence for the first series and dominated convergence for the second, using the Beckner--Onofri bound,
\begin{align}
        \F_K(q\sg)-\F_K(\sg)
        ={}&
        \left[
        \Ent(q\sg\mid\sg)
        -\frac12\sum_{\ell=1}^\infty
        b_{\ell,d}\norm{\Pi_\ell(q-\qu)}_2^2
        \right] \notag\\
        &+\frac12\sum_{\ell=1}^\infty
        \left(b_{\ell,d}-2K\lambda_\ell\right)
        \norm{\Pi_\ell(q-\qu)}_2^2
        \ge0.
        \label{eq:ordinary-gap-decomposition}
\end{align}
Thus, $\sg$ is a global minimizer for every $0\le K\le K_\#$.

If $K<K_\#$ and $q\ne\qu$ {is also a minimizer of $\F_K$}, then equality must hold in \eqref{eq:ordinary-gap-decomposition}.  But for $K<K_\#=K_1$, all coefficients $b_{\ell,d}-2K\lambda_\ell$ are strictly positive: this is immediate in degree one, follows from $2K\lambda_\ell<2K_1\lambda_\ell\le b_{\ell,d}$ when $\lambda_\ell>0$, and is trivial when $\lambda_\ell\le0$.  Hence, equality forces $\Pi_\ell(q-\qu)=0$ for every $\ell\ge1$, and therefore $q=\qu$.  Thus, the minimizer is unique for $K<K_\#$.

{It remains to prove uniqueness at $K=K_\#=K_1$.}  Since $2K_1\lambda_\ell=r_\ell$, \eqref{eq:ordinary-gap-decomposition} gives
\begin{align}
        \F_{K_1}(q\sg)-\F_{K_1}(\sg)
        ={}&
        \left[
        \Ent(q\sg\mid\sg)
        -\frac12\sum_{\ell=1}^\infty b_{\ell,d}
        \norm{\Pi_\ell(q-\qu)}_2^2
        \right] \notag\\
        &+\frac12\sum_{\ell=1}^\infty
        (b_{\ell,d}-r_\ell)
        \norm{\Pi_\ell(q-\qu)}_2^2.
        \label{eq:deficit-bound}
\end{align}
If $q\sg$ is also a minimizer at $K_1$, then equality in \eqref{eq:deficit-bound} implies equality in \cref{prop:dualBO} and
\begin{equation}\label{eq:equality-projections}
        (b_{\ell,d}-r_\ell)\norm{\Pi_\ell(q-\qu)}_2^2=0,
        \qquad\forall\ell\ge1.
\end{equation}
By the equality case in \cref{prop:dualBO}, $q=J_a$ for some $|a|<1$.  If $a\ne0$, then \cref{lem:nonzero-conformal} says that every higher harmonic projection is nonzero, contradicting strictness in at least one degree $\ell\ge2$.  Hence, $a=0$ and $q=\qu$.  The continuity assertion follows from \cref{lem:Kc}.
\end{proof}

\section{Noisy-transformer coefficients and the continuous regime}\label{sec:transformer-coefficients}

The Beckner--Onofri criterion reduces the continuous side to the degree-one-leading condition together with the comparison $r_\ell\le b_{\ell,d}$ for all $\ell\ge2$, with strict inequality in at least one higher degree.  We now compute the ratios $r_\ell$ for the attention kernel $W_\beta$ defined in \eqref{eq:intro-centered-kernel} and prove exactly when that comparison holds.

Since the subtraction of $m_d(\beta)/\beta$ in \eqref{eq:intro-centered-kernel} affects only the zeroth Funk--Hecke mode, the nonzero eigenvalues of $W_\beta$ agree with those of $\beta^{-1}e^{\beta x\cdot y}$.  Below, $\lambda_\ell(\beta)$ denotes this eigenvalue for $\ell\ge1$.

\begin{lemma}[Funk--Hecke coefficients]\label{lem:transformer-coefficients}
For every $\ell\ge1$,
\begin{equation}\label{eq:lambda-transformer}
        \lambda_\ell(\beta)
        =\frac{\Gamma(d/2)}{\beta}\left(\frac{2}{\beta}\right)^{(d-2)/2}
        I_{\ell+(d-2)/2}(\beta).
\end{equation}
Consequently,
\begin{equation}\label{eq:transformer-Kone}
        K_1^{(d)}(\beta)=\frac{1}{2\lambda_1(\beta)}
        =\frac{\beta^{d/2}}{2^{d/2}\Gamma(d/2)I_{d/2}(\beta)},
\end{equation}
and
\begin{equation}\label{eq:transformer-ratios}
        r_\ell(\beta)=\frac{\lambda_\ell(\beta)}{\lambda_1(\beta)}
        =\frac{I_{\ell+d/2-1}(\beta)}{I_{d/2}(\beta)}.
\end{equation}
\end{lemma}

\begin{proof}
For a fixed spherical harmonic $Y\in\Hh_\ell$, the Funk--Hecke formula \cite[Section~2.6]{AtkinsonHan2012} gives
\begin{equation}\label{eq:FH-exponential-eigen}
        \int_{\Sd} e^{\beta x\cdot y}Y(y)\,d\sg(y)
        =\Lambda_\ell(\beta)Y(x).
\end{equation}
Here, $\Lambda_\ell(\beta)$ denotes the Funk--Hecke coefficient for the raw exponential kernel $e^{\beta x\cdot y}$.
With normalized surface measure, the exponential Funk--Hecke coefficient is \cite[Section 2.6]{AtkinsonHan2012}
\begin{equation}\label{eq:exponential-FH}
        \Lambda_\ell(\beta)=\Gamma(d/2)\left(\frac2\beta\right)^{(d-2)/2}
        I_{\ell+(d-2)/2}(\beta).
\end{equation}
Dividing by $\beta$ gives \eqref{eq:lambda-transformer}.  The formulas \eqref{eq:transformer-Kone} and \eqref{eq:transformer-ratios} follow immediately.
\end{proof}

\begin{remark}
For $d=2$, \eqref{eq:lambda-transformer} reduces to the circle coefficient
\begin{equation}\label{eq:d2-transformer-coefficient}
        \lambda_\ell(\beta)=\frac{I_\ell(\beta)}{\beta}.
\end{equation}
\end{remark}

The coefficient formula gives the degree-one threshold for the uniform measure.  The following comparison with the Beckner--Onofri weights determines whether that threshold is also the global-minimizer threshold.

\begin{lemma}[Bessel-ratio comparison]\label{lem:bessel-comparison}
For every $d\ge2$, equation \eqref{eq:beta-star-main} has a unique solution $\beta_*^{(d)}>0$.  It satisfies the elementary bounds
\begin{align}
        1+\frac2d&<\beta_*^{(d)}<1+\frac2d+\frac4{d^2},
\end{align}
and the strict monotonicity property
\begin{align}
        \beta_*^{(d+1)}&<\beta_*^{(d)},\qquad d\ge2.
\end{align}
As $d\to\infty$,
\begin{equation}\label{eq:beta-star-large-d}
        \beta_*^{(d)}=1+\frac2d+\frac1{d^2}+O(d^{-3}).
\end{equation}
If $0<\beta\le\beta_*^{(d)}$, then
\begin{equation}\label{eq:all-mode-majorization}
        r_\ell(\beta)\le b_{\ell,d}
        \qquad\text{for every }\ell\ge2,
\end{equation}
with strict inequality in at least one degree $\ell\ge2$.  Moreover, $r_\ell(\beta)<1$ for every $\ell\ge2$ and every $\beta>0$.
\end{lemma}

\begin{proof}
We use standard Bessel-ratio facts; see \cite{Segura2021} for the monotonicity properties and \cite[Section~10.30]{NISTDLMF} for the endpoint asymptotics used below.  Define
\begin{equation}\label{eq:bessel-single-ratio-map}
        \rho_\alpha(x):=\frac{I_{\alpha+1}(x)}{I_\alpha(x)},\qquad \alpha\ge0.
\end{equation}
For each $\alpha\ge0$, the map $x\mapsto\rho_\alpha(x)$ is strictly increasing on $(0,\infty)$ and satisfies
\begin{equation}\label{eq:bessel-ratio-limits}
        \rho_\alpha(x)\to0\quad(x\downarrow0),
        \qquad
        \rho_\alpha(x)\to1\quad(x\to\infty).
\end{equation}
Set $\nu=d/2$.  For each $\alpha\ge1$, let $\beta_\alpha$ be defined by $\rho_\alpha(\beta_\alpha)=1/(2\alpha)$.  Since $0<1/(2\alpha)<1$, the preceding monotonicity and endpoint limits prove existence and uniqueness; moreover $\beta_*^{(d)}=\beta_\nu$.

For the elementary bounds and monotonicity, the continued fraction for modified-Bessel ratios \cite[Section~10.33]{NISTDLMF} gives, for $x>0$,
\begin{align}
        \frac{x}{2(\nu+1)+x^2/(2(\nu+2))}
        &<\rho_\nu(x)<\frac{x}{2(\nu+1)}.
\end{align}
Together with the strict monotonicity of $x\mapsto\rho_\nu(x)$, this implies
\begin{align}
        1+\frac1\nu&<\beta_\nu<1+\frac1\nu+\frac1{\nu^2}.
\end{align}
Indeed, the upper bound on $\rho_\nu$ gives $\rho_\nu(1+1/\nu)<1/(2\nu)$, while the lower bound gives $\rho_\nu(1+1/\nu+1/\nu^2)>1/(2\nu)$.  Since $\beta_*^{(d)}=\beta_\nu$ and $\nu=d/2$, this gives the displayed bounds for $\beta_*^{(d)}$.

For monotonicity, apply the lower bound with order $\nu+1/2$ at $x=1+1/\nu$.  Since
\begin{align}
        (2\nu+1)\left(1+\frac1\nu\right)-(2\nu+3)
        -\frac{(1+1/\nu)^2}{2\nu+5}
        &=\frac{\nu^2+3\nu-1}{\nu^2(2\nu+5)}>0,
\end{align}
we have
\begin{align}
        \rho_{\nu+1/2}\!\left(1+\frac1\nu\right)&>\frac1{2\nu+1}.
\end{align}
By strict monotonicity,
\begin{align}
        \beta_{\nu+1/2}&<1+\frac1\nu<\beta_\nu.
\end{align}
Since $\beta_*^{(d)}=\beta_{d/2}$, this gives $\beta_*^{(d+1)}<\beta_*^{(d)}$ for every $d\ge2$.

We next record the elementary large-$d$ asymptotic of this solution.  With this notation, for $\beta$ in a fixed compact subset of $(0,\infty)$, the power-series representation of $I_\nu$ gives, uniformly in $\beta$,
\begin{equation}\label{eq:bessel-large-order-ratio}
        2\nu\rho_\nu(\beta)
        =\beta\left(1-\frac1\nu+\frac{1-\beta^2/4}{\nu^2}+O(\nu^{-3})\right).
\end{equation}
Indeed, writing
\begin{equation}\label{eq:bessel-series-factor}
        I_\nu(\beta)=\frac{(\beta/2)^\nu}{\Gamma(\nu+1)}
        \sum_{k=0}^\infty\frac{(\beta^2/4)^k}{k!(\nu+1)_k},
\end{equation}
one obtains
\begin{equation}\label{eq:bessel-series-quotient}
        \frac{\sum_{k\ge0}(\beta^2/4)^k/[k!(\nu+2)_k]}
             {\sum_{k\ge0}(\beta^2/4)^k/[k!(\nu+1)_k]}
        =1-\frac{\beta^2}{4\nu^2}+O(\nu^{-3}),
\end{equation}
and \eqref{eq:bessel-large-order-ratio} follows after multiplying by $\beta\nu/(\nu+1)$.  Applying the same expansion at $\beta=1$ and at $\beta=1+2/\nu$, and using monotonicity of $\rho_\nu$, first gives $\beta_\nu=1+O(\nu^{-1})$.  Substituting
\begin{equation}\label{eq:beta-star-large-d-ansatz}
        \beta=1+\frac A\nu+\frac B{\nu^2}+O(\nu^{-3})
\end{equation}
into the defining equation $2\nu\rho_\nu(\beta)=1$ gives $A=1$ and $B=1/4$.  Since $\beta_*^{(d)}=\beta_\nu$ and $\nu=d/2$, this proves \eqref{eq:beta-star-large-d}.

The adjacent ratios are decreasing in the order parameter: for $\alpha\ge0$, $j\ge0$, and $\beta>0$,
\begin{equation}\label{eq:bessel-order-monotonicity}
        \rho_{\alpha+j+1}(\beta)\le\rho_\alpha(\beta).
\end{equation}
Equivalently, $I_{\alpha+1}(\beta)^2\ge I_\alpha(\beta)I_{\alpha+2}(\beta)$, with strict inequality for $\beta>0$.  These single-ratio and double-ratio/Turan monotonicity properties are contained in Segura's results on modified Bessel functions \cite{Segura2021}.

Multiplying adjacent ratios yields, for $\ell\ge2$,
\begin{equation}\label{eq:transformer-ratio-bound}
        r_\ell(\beta)
        =\prod_{j=0}^{\ell-2}\rho_{\nu+j}(\beta)
        \le \rho_\nu(\beta)^{\ell-1}.
\end{equation}
If $\beta\le\beta_*^{(d)}=\beta_\nu$, then $\rho_\nu(\beta)\le1/(2\nu)=1/d$, and \eqref{eq:transformer-ratio-bound} gives
\begin{equation}\label{eq:d-power-bound}
        r_\ell(\beta)\le d^{-(\ell-1)}.
\end{equation}
On the other hand, since $j/(d+j-1)\ge1/d$ for every $j\ge1$, we have
\begin{equation}\label{eq:d-power-vs-beckner}
        d^{-(\ell-1)}
        \le \prod_{j=1}^{\ell-1}\frac{j}{d+j-1}
        = b_{\ell,d}.
\end{equation}
This proves \eqref{eq:all-mode-majorization}.  If $\beta<\beta_*^{(d)}$, the inequality is strict already in degree $2$.  If $\beta=\beta_*^{(d)}$, degree $3$ is strict because
\begin{equation}\label{eq:endpoint-strict-degree-three}
        r_3(\beta_*^{(d)})\le d^{-2}<\frac{2}{d(d+1)}=b_{3,d}.
\end{equation}
Finally, \eqref{eq:bessel-ratio-limits} and monotonicity imply $\rho_\alpha(\beta)<1$ for $\alpha\ge0$ and $\beta>0$, hence $r_\ell(\beta)<1$ for every $\ell\ge2$.
\end{proof}

\section{The degree-two quartic obstruction}\label{sec:quartic}

The continuous side is governed by all-mode Beckner--Onofri majorization.  The discontinuous side is detected by the first failure of that majorization, namely the crossing of the degree-two ratio past
$b_{2,d}=\frac1d$.
The terminology ``degree-two quartic obstruction'' refers to the following perturbative mechanism.  At $K=K_1$, the degree-one quadratic variation is neutral.  Perturbing only in a degree-one direction is therefore inconclusive.  If one also adds the degree-two component generated by the square of a degree-one harmonic, then after the quadratic term cancels, the first decisive coefficient is of order four in the perturbation size.

\begin{lemma}[Degree-two quartic obstruction]\label{lem:quartic}
Let $W$ be continuous, symmetric, centered, and zonal, and assume $\lambda_1>0$.  If
\begin{equation}\label{eq:r2-obstruction}
        r_2=\frac{\lambda_2}{\lambda_1}>\frac1d,
\end{equation}
then $\sg$ is not a global minimizer of $\F_{K_1}$.  In particular, $K_c<K_1$.
\end{lemma}

\begin{proof}
Replace $W$ by $W/(2\lambda_1)$.  This normalization gives $\lambda_1=1/2$ and $K_1=1$, without changing $r_2$.  This rescaling puts the endpoint at $K_1=1$ and the neutral degree-one eigenvalue at $1/2$, which simplifies the quartic expansion while preserving the condition $r_2>1/d$. 

Let
\begin{equation}\label{eq:quartic-test-functions}
        \varphi(x)=\sqrt d\,x_1,
        \qquad
        h=\varphi^2-1=dx_1^2-1.
\end{equation}
Then $\varphi\in\Hh_1$, and $h\in\Hh_2$ because it is the restriction to $\Sd$ of the harmonic homogeneous quadratic polynomial $dX_1^2-|X|^2$ on $\R^d$.  Moreover,
\begin{equation}\label{eq:quartic-moments}
        \int_{\Sd}\varphi^4\,d\sg=\frac{3d}{d+2},
        \qquad
        \norm{h}_2^2=\int_{\Sd}\varphi^2h\,d\sg=\frac{2(d-1)}{d+2}=:A.
\end{equation}
For fixed $c\in\R$ and $|\ep|$ sufficiently small,
\begin{equation}\label{eq:quartic-density}
        q_{\ep,c}=1+\ep\varphi+c\ep^2h
\end{equation}
is a positive probability density.  We use the Taylor expansion
\begin{equation}\label{eq:entropy-taylor-series}
        (1+s)\log(1+s)=s+\frac12s^2-\frac16s^3+\frac1{12}s^4+O(s^5).
\end{equation}
Since $\varphi$ is odd under the reflection $x_1\mapsto -x_1$, $\int\varphi^3\,d\sg=0$.  Writing $s=\ep\varphi+c\ep^2h$, the $\ep^4$ coefficient in the Taylor approximation of the entropy is
\begin{align}
\frac12c^2\norm{h}_2^2-\frac12c\int\varphi^2h\,d\sg+\frac1{12}\int\varphi^4\,d\sg.
\end{align}
Using the moment identities above, we obtain the relative entropy expansion
\begin{equation}\label{eq:entropy-quartic}
        \Ent(q_{\ep,c}\sg\mid\sg)
        =\frac12\ep^2+\left[\frac A2(c^2-c)+\frac{d}{4(d+2)}\right]\ep^4+O(\ep^5).
\end{equation}
Since $W$ is zonal, $T_W$ acts by a scalar on each $\Hh_\ell$, and orthogonality removes cross terms between the degree-one and degree-two components.  The interaction energy at $K_1=1$ is therefore exactly
\begin{equation}\label{eq:energy-quartic}
        \ip{q_{\ep,c}-\qu}{T_W(q_{\ep,c}-\qu)}
        =\frac12\ep^2+\lambda_2Ac^2\ep^4.
\end{equation}
Combining \eqref{eq:entropy-quartic} and \eqref{eq:energy-quartic}, the $\ep^2/2$ terms cancel and we obtain
\begin{equation}\label{eq:gap-quartic}
        \F_1(q_{\ep,c}\sg)-\F_1(\sg)
        =\left[A\left(\frac12-\lambda_2\right)c^2-\frac A2c+\frac{d}{4(d+2)}\right]\ep^4+O(\ep^5).
\end{equation}
If $\lambda_2<1/2$, optimizing the bracket over $c$ gives
\begin{equation}\label{eq:quartic-minimum}
        \frac{1}{4(d+2)}\left[d-\frac{d-1}{1-2\lambda_2}\right],
\end{equation}
which is negative exactly when $2\lambda_2>1/d$, i.e.\ when $r_2>1/d$ in the present normalization.  If $\lambda_2\ge1/2$, the bracket in \eqref{eq:gap-quartic} is made negative by any sufficiently large choice of $c$ (recall that $A>0$).  Thus, for sufficiently small nonzero $\ep$, there is a finite-entropy competitor with strictly smaller free energy than $\sg$ at $K_1$.  The strict inequality implies positive centered interaction energy, and the same monotonicity argument as in \cref{lem:Kc}(iii) gives $K_c<K_1$.
\end{proof}

\section{Proof of the main result}\label{sec:proof-main}

We now assemble the four inputs to prove the main result \cref{thm:main}: the coefficient computation, the Bessel-ratio comparison, the Beckner--Onofri continuity criterion, and the degree-two quartic obstruction.

\begin{proof}[Proof of \cref{thm:main}]
The existence and uniqueness of $\beta_*^{(d)}$ are proved in \cref{lem:bessel-comparison}.  {The formula for $K_1^{(d)}(\beta)$ follows from} \cref{lem:transformer-coefficients}.  The same lemma gives
\begin{equation}\label{eq:main-proof-ratio-formula}
        r_\ell(\beta)=\frac{I_{\ell+d/2-1}(\beta)}{I_{d/2}(\beta)}.
\end{equation}
By \cref{lem:bessel-comparison}, $r_\ell(\beta)<1$ for every $\ell\ge2$ and $\beta>0$, implying  {$K_\#^{(d)}(\beta)=K_1^{(d)}(\beta)$.}

Assume first that $0<\beta\le\beta_*^{(d)}$.  Then \cref{lem:bessel-comparison} gives
\begin{equation}\label{eq:main-proof-majorization}
        r_\ell(\beta)\le b_{\ell,d}
        \qquad\text{for every }\ell\ge2,
\end{equation}
with strict inequality in at least one degree.  The Beckner--Onofri continuity criterion, \cref{thm:BOcriterion}, therefore implies that $K_c^{(d)}(\beta)=K_\#^{(d)}(\beta)=K_1^{(d)}(\beta)$, that $\sg$ is the unique global minimizer for every $0\le K\le K_\#^{(d)}(\beta)$, and that the transition is continuous.

Assume next that $\beta>\beta_*^{(d)}$.  Since $\beta\mapsto I_{d/2+1}(\beta)/I_{d/2}(\beta)$ is strictly increasing, the degree-two ratio satisfies
\begin{equation}\label{eq:main-proof-r2-large}
        r_2(\beta)=\frac{I_{d/2+1}(\beta)}{I_{d/2}(\beta)}>\frac1d.
\end{equation}
The quartic obstruction \cref{lem:quartic} gives a finite-entropy density with strictly lower free energy than \(\sg\) at \(K=K_1^{(d)}(\beta)=K_\#^{(d)}(\beta)\).  Applying \cref{lem:Kc}(iii), we obtain $K_c^{(d)}(\beta)<K_\#^{(d)}(\beta)$,
and the transition is discontinuous in the sense of \cref{def:transition}.

In all cases, $K_c^{(d)}(\beta)<\infty$: for $K>K_\#^{(d)}(\beta)=K_1^{(d)}(\beta)$, the degree-one second variation at $\qu$ is negative, so $\sg$ is not even a local minimizer.  The critical-point uniqueness assertion therefore follows by applying \cref{prop:critical-uniqueness-half} to the kernel $W=W_\beta$ defined in \eqref{eq:intro-centered-kernel}.  In the continuous regime, the first part of the theorem identifies $K_c^{(d)}(\beta)$ with $K_\#^{(d)}(\beta)$, giving uniqueness of critical points for $0\le K\le K_\#^{(d)}(\beta)/2$.
\end{proof}

\bibliographystyle{alpha}
\bibliography{noisy_transformer_sphere_refs}

\end{document}